\documentclass[12pt]{article}
\usepackage{tikz}
\usetikzlibrary{calc,through,backgrounds}
\usepackage{multicol}

\usepackage{bm,amssymb,amsthm,amsfonts,amsmath,latexsym,cite,color, psfrag,graphicx,ifpdf,pgfkeys,pgfopts,xcolor,extarrows}

\newtheorem{theo}{Theorem}[section]

\newtheorem{lem} [theo]{Lemma}

\newtheorem{re}[theo]{Remark}

\allowdisplaybreaks

\makeatletter \@addtoreset{equation}{section}
\@addtoreset{theo}{}\makeatother

\usepackage{hyperref}

\begin{document}
\begin{center}
{\Large\bf Upper bounds of dual flagged Weyl characters}

\vskip 6mm
{\small  Zhuowei Lin, Simon C.Y. Peng,  Sophie C.C. Sun }

\end{center}

\begin{abstract}

For a  subset $D$  of boxes in  an  $n\times n$ square grid,  let $\chi_{D}(x)$ denote
the dual character of the flagged Weyl module associated to $D$. It is known that $\chi_{D}(x)$
specifies to a Schubert polynomial (resp.,  a key polynomial) in the case
when $D$ is the Rothe diagram of a permutation (resp., the skyline diagram of a composition).
One can naturally define a lower  and an upper bound of  $\chi_{D}(x)$.
M{\'e}sz{\'a}ros, St. Dizier and Tanjaya conjectured that
 $\chi_{D}(x)$ attains the upper bound if and only if $D$ avoids a certain subdiagram.
We provide a proof of  this conjecture.

\end{abstract}

\vskip 3mm


\vskip 3mm


\section{Introduction}

We adopt the notation $[n]=\{1,2,\ldots, n\}$. A diagram  in a square grid $[n]\times [n]$ means a subset $D$ of boxes in $[n]\times [n]$. To a diagram $D$, the associated flagged Weyl module $\mathcal{M}_{D}$ is a module of the group $B$ of invertible upper-triangular matrices over $\mathbb{C}$ \cite{kraskiewicz1987foncteurs, kraskiewicz2004schubert, magyar1998schubert}.  The dual character of $\mathcal{M}_{D}$, denote $\chi_{D}(x)$, is a polynomial in $x_1,\ldots, x_n$, which specifies to the Schubert polynomial $\mathfrak{S}_w(x)$ when $D$ is the Rothe diagram of a permutation $w$, and to the key polynomial $\kappa_\alpha(x)$ when $D$ is the skyline diagram of a composition $\alpha$. The flagged Weyl modules have recently served as an important tool in the study of combinatorial properties of Schubert polynomials or key polynomials, see for example \cite{FG-1, fan2022upper, 2018Schubert, fink2021zero, meszaros2021principal}.

By the construction of flagged Weyl modules, it is natural to define a lower bound and  an upper bound of the dual character $\chi_{D}(x)$. When $\chi_{D}(x)$ is a Schubert polynomial (that is, $D$ is the Rothe diagram of a permutation),  Fink,  M{\'e}sz{\'a}ros  and  St. Dizier \cite{fink2021zero} found a criterion of when $\chi_{D}(x)$ attains the lower bound. When  $\chi_{D}(x)$ is a key polynomial (that is, $D$ is the skyline diagram of a composition), a criterion of when $\chi_{D}(x)$ attains the lower bound was given  by Hodges and Yong \cite{hodges2023multiplicity}. In the opposite direction, Fan and Guo \cite{fan2022upper} proved an upper bound criterion  in the case when $\chi_{D}(x)$ is a Schubert polynomial or a key polynomial.
For a general diagram $D$, M{\'e}sz{\'a}ros,   St. Dizier  and   Tanjaya \cite[Conjecture 29]{meszaros2021principal} conjectured that $\chi_{D}(x)$ reaches the upper bound if and only if $D$ avoids a certain single subdiagram. The aim of this paper is to prove this conjecture.

To state the conjecture of M{\'e}sz{\'a}ros,   St. Dizier  and   Tanjaya, let us give a brief overview of the flagged Weyl module $\mathcal{M}_{D}$ for a diagram $D\subseteq [n]\times [n]$.
 We use $(i, j)$ to denote the box in row $i$ and column $j$ in the matrix coordinate. We may alternatively represent $D$ by a sequence
 $D=(D_{1}, D_{2},\ldots, D_{n})$ of subsets of $[n]$. Precisely, $i\in D_j$
 if and only if the box $(i,j)$ belongs to $D$. For example, the diagram in Figure \ref{fig:enter-label-1122} can be represented as $D=(\{2,3,4\},\emptyset,\{1,2\},\{3\})$.
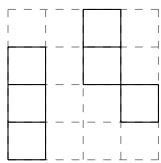
\begin{figure}[ht]
    \centering
    \begin{align*}
	\begin{tikzpicture}
		\centering
		\draw[dashed,step=0.5,help lines] (-1,-1) grid (1,1);
		\draw[black,step=0.5] (-1,-1) grid (-0.5,0.5);
		\draw[black,step=0.5] (0,0) grid (0.5,1);
		\draw[black,step=0.5] (0.5,-0.5) grid (1,0);
	\end{tikzpicture}
\end{align*}
    \caption{A diagram in $[4]\times [4]$.}
    \label{fig:enter-label-1122}
\end{figure}

For two subsets $R=\{r_1<r_2<\cdots<r_m\}$ and $S= \{s_1<s_2<\cdots<s_m\} $ of $[n]$ with the same cardinality, we write $R\leq S$ (called the Gale order) if  $r_i\leq s_i$ for $1\leq i\leq m$.
For two diagrams $C=(C_{1},C_{2},\ldots, C_{n})$ and $D=(D_{1},D_{2},\ldots,D_{n})$, we say $C\leq D$ if $C_{j}\leq D_{j}$ for each $1\leq j\leq n$.

Let $\mathrm{GL}(n,\mathbb{C})$ be the general linear group of $n\times n$ invertible matrices on $\mathbb{C}$, and $B$  the Borel subgroup of $\mathrm{GL}(n,\mathbb{C})$ consisting of all the upper-triangular matrices. Let $Y$ denote the upper-triangular matrix of variables $y_{ij}$ ($1\leq i\leq j\leq n$):
\begin{equation*}
	Y=
	\begin{bmatrix}
		y_{11}&y_{12}&\ldots&y_{1n}\\
		0&y_{22}&\ldots&y_{2n}\\
		\vdots&\vdots&\ddots&\vdots\\
		0&0&\ldots&y_{nn}
	\end{bmatrix}.
\end{equation*}
Write  $\mathbb{C} [Y]$ for the space of polynomials over $\mathbb{C}$ in the variables  $\{y_{ij}\}_{i\leq j}$. Define the (right) action of $B$   on $\mathbb{C} [Y]$ by   $f(Y)\cdot b=f(b^{-1}\cdot Y)$, where $b\in B$ and  $f\in \mathbb{C} [Y]$. For two subsets $R$ and $S$ of $[n]$ with the same cardinality, we use $Y^{R}_{S}$ to denote the submatrix of $Y$ obtained by restricting to rows indexed by $R$ and columns indexed by $S$. It is easily checked that $Y^{R}_{S}\neq 0$ if and only if $R\leq S$.
For a diagram $D=(D_{1},D_{2},\ldots, D_{n})$, denote
\[
\mathrm{det}\left(Y_{D}^C\right)=\prod_{j=1}^{n}\mathrm{det}\left(Y_{D_{j}}^{C_{j}}\right).
\]
The flagged Weyl module $\mathcal{M}_{D}$  for  $D$ is  a subspace of $\mathbb{C} [Y]$ defined by
\begin{align*}	\mathcal{M}_{D}=\mathrm{Span}_{\mathbb{C}}\left\{\mathrm{det}\left(Y_D^C\right)\colon C\leq D\right\},
\end{align*}
which  is a $B$-module with the action inherited from the action of $B$ on $\mathbb{C} [Y]$.

The character of $\mathcal{M}_{D}$ is defined as
\[\mathrm{char}(\mathcal{M}_{D})(x_{1},\ldots, x_{n})=\mathrm{tr}(X\colon \mathcal{M}_{D}\rightarrow \mathcal{M}_{D}),\]
where $X$ is the diagonal matrix with diagonal entries $x_{1},\ldots,x_{n}$. Its  dual character is of the form
\begin{align*}
\chi_{D}(x):=\mathrm{char}(\mathcal{M}_{D})(x_{1}^{-1},\ldots,x_{n}^{-1}).
\end{align*}
It is readily checked that the polynomial $\mathrm{det}\left(Y_{D}^C\right)$  is an
eigenvector of $X$ with eigenvalue
\[\prod_{j=1}^n\prod_{i\in C_j}x_i^{-1}.\]
So the set of monomials appearing in $\chi_{D}(x)$ is exactly
\[\left\{x^C\colon C\le D\right\},\]
where
\[
x^{C}=\prod_{j=1}^{n}\prod_{i\in C_{j}} x_{i}.
\]
Specifically, the coefficient of a monomial $x^a$
appearing in $\chi_{D}(x)$ equals the dimension of the  eigenspace
\begin{align}\label{UUUP}
\mathrm{Span}_{\mathbb{C}}\left\{\mathrm{det}\left(Y_{D}^C\right)\colon  C\leq D, \, x^C=x^a\right\}.
\end{align}
By the above observations, we obtain
the coefficient-wise inequality
\begin{equation}\label{es}
\sum_{x^a\in \{x^C\colon C\leq D\}}x^a
\leq \chi_{D}(x)
\leq
\sum_{ C\leq D} x^C,
\end{equation}
which can be equivalently described as
\[
\# \{x^C\colon C\leq D\} \leq
\chi_{D}(1,\ldots,1)
\leq \# \{C\colon C\leq D\}.
\]

If the equality on left-hand  side of \eqref{es} holds, then we say that $\chi_{D}(x) $ attains the lower bound. In this case, $\chi_{D}(x) $ is called  zero-one or multiplicity-free, that is, the coefficient of $x^a$ in $\chi_{D}(x)$ is equal to either zero or one. As aforementioned, when $\chi_{D}(x)$ is a Schubert polynomial or a key polynomial,
the criterion of when $\chi_{D}(x)$ attains the lower bound was given respectively by Fink,  M{\'e}sz{\'a}ros  and  St. Dizier \cite{fink2021zero} and
  Hodges and Yong \cite{hodges2023multiplicity}.
 A conjectured criterion  by M{\'e}sz{\'a}ros,   St. Dizier  and   Tanjaya  for general $D$ \cite[Conjecture 27]{meszaros2021principal} is still open.

This paper focuses on the upper bound situation, namely, the case when the equality on the right-hand side of \eqref{es} holds.
Equivalently, for any monomial $x^a$ appearing in  $\chi_{D}(x)$, the polynomials generating the eigenspace in \eqref{UUUP} are all linearly independent.
When $\chi_{D}(x)$ is a Schubert polynomial or a key polynomial, the corresponding criterion was found  by
Fan and Guo\cite{fan2022upper} (more generally, for all northwest diagrams).
We shall prove a criterion of when $\chi_{D}(x)$ attains the upper bound for any given diagram $D$, as conjectured by
M{\'e}sz{\'a}ros,   St. Dizier  and   Tanjaya \cite[Conjecture 29]{meszaros2021principal}.

\begin{theo}[{\cite[Conjecture 29]{meszaros2021principal}}]\label{conj}
The dual character 	$\chi_{D}(x)$ attains the upper bound if and only if the diagram $D$ does not contain any subdiagram as shown in Figure  \ref{fig:enter-label-Y1}.
    \begin{figure}[ht]
        \centering
        \begin{equation*}\label{2.1}
		\begin{tikzpicture}
			\centering
			\draw[dashed, black] (-1.75,2) -- (1.75,2);
			\draw[dashed, black] (-1.75,1.5) -- (1.75,1.5);
			\draw[dashed, black] (-1.75,0) -- (1.75,0);
			\draw[dashed, black] (-1.75,0.5) -- (1.75,0.5);
			\draw[dashed, black] (-0.5,-0.5) -- (-0.5,2.5);
			\draw[dashed, black] (-1,-0.5) -- (-1,2.5);
			\draw[dashed, black] (1,-0.5) -- (1,2.5);
			\draw[dashed, black] (0.5,-0.5) -- (0.5,2.5);
			\draw[black] (-1,1.5) rectangle (-0.5,2);
			\draw[black] (0.5,1.5) rectangle (1,2);
			\draw[black] (-1,1.5) -- (-0.5,2);
			\draw[black] (0.5,1.5) -- (1,2);
			\draw[black] (-1,2) -- (-0.5,1.5);
			\draw[black] (0.5,2) -- (1,1.5);
            \node at (-2.25,1.75) {$i_1$};
            \node at (-2.25,0.25) {$i_2$};
            \node at (-0.75,2.85) {$j_1$};
            \node at (0.75,2.85) {$j_2$};
			\filldraw [lightgray] (-1,0) rectangle (-0.5,0.5);
			\draw[black] (-1,0) rectangle (-0.5,0.5);
			\filldraw [lightgray] (0.5,0) rectangle (1,0.5);
			\draw[black] (0.5,0) rectangle (1,0.5);
		\end{tikzpicture}
	\end{equation*}
        \caption{A subdiagram that $D$ should  avoid.}
        \label{fig:enter-label-Y1}
    \end{figure}
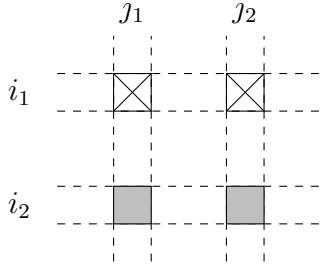
    In other words, there do not exist $1\leq i_1<i_2\leq n$ and $1\leq j_1<j_2\leq n$ such that the two boxes $(i_1, j_1)$ and $(i_1, j_2)$ do not belong to $D$, while the two boxes
$(i_2, j_1)$ and $(i_2, j_2)$  belong to $D$.
\end{theo}

\begin{re}
 Huh,   Matherne,   M{\'e}sz{\'a}ros  and   St. Dizier \cite{Huh} asked if  the normalized dual character Weyl character $\mathrm{N}(\chi_{D}(x))$ is Lorentzian, where $\mathrm{N}$ is a linear operator sending a monomial $x^a=x^{a_1}\cdots x^{a_n}$ to
 \[
 \mathrm{N}(x^a)=\frac{x^a}{a_1!\cdots a_n!}.
 \]
 Invoking   Theorem \ref{conj} and using  the same arguments as in the proof of \cite[Proposition 17]{Huh}, we can prove that if $D$ avoids the configuration  in Figure  \ref{fig:enter-label-Y1}, then
 $\chi_{D}(x)$  is Lorentzian, and thus $\mathrm{N}(\chi_{D}(x))$ is Lorentzian.
 Here we used the fact that  the Lorentzian
property of a polynomial $ f(x)$   implies that of  $\mathrm{N}(f(x))$ \cite[Corollary 3.7]{BH}.

\end{re}

\section{Sufficiency of Theorem \ref{conj}}

In this section, prove the sufficiency of Theorem \ref{conj}. Let us begin with establishing a combinatorial model generating the polynomial $\mathrm{det}\left(Y_{D}^C\right)$ for any $C\leq D$.
A  {\it flagged filling} of a diagram  $D=(D_1,\ldots, D_n)$   is a filling of the boxes of $D$ with positive integers   such that
\begin{itemize}
		\item[(1)] each box of $D$ receives exactly one integer, and the entries in each column are distinct;
		
  \item[(2)] the entry in the box $(i,j)$ cannot exceed $i$.
\end{itemize}
	
\begin{re}
For a flagged filling  $F\in \mathcal{F}_{D}$,  if letting $C=(C_1,\ldots, C_n)$ be the diagram with $C_j$ being the set of entries in the $j$-th column of $F$, then it is not hard to check that  $C\leq D$.
\end{re}

Let $\mathcal{F}_{D}$ denote the set of all flagged fillings of $D$. For   $F\in \mathcal{F}_{D}$, define the inversion number $\mathrm{inv}(F)$ of $F$ as follows.
Assume that  $F_j$ is $j$-th column of   $F$. Let $a=a_1\cdots a_m$ be the word obtained by reading the entries of $F_j$ from top to bottom. Define  $\mathrm{inv}(F_i)$ to be  the inversion number of $a$,
namely,
\[
\mathrm{inv}(F_i)=\#\{(r, s)\colon 1\leq r<s\leq m, \ a_r>a_s\}.
\]
Set
\[
\mathrm{inv}(F)=\mathrm{inv}(F_1)+\cdots +\mathrm{inv}(F_n).
\]
For example, consider the flagged filling depicted  in Figure \ref{fig:enter-label-UU22}.
\begin{figure}[ht]
    \centering
    \begin{align*}
	F^{}=\begin{tikzpicture}[baseline={([yshift=-.5ex]current bounding box.center)}]
		\centering
		\draw[dashed,step=0.5,help lines] (-1,-1) grid (1,1);
		\draw[black,step=0.5] (-1,-1) grid (-0.5,0.5);
        \draw[black,step=0.5] (-0.5,0) rectangle (0,0.5);
        \draw[black,step=0.5] (-0.5,0) rectangle (0,-0.5);
        \draw[black,step=0.5] (-0.5,-1) rectangle (0,-0.5);
		\draw[black,step=0.5] (0,1) grid (0.5,0.5);
		\draw[black,step=0.5] (0.5,-1) grid (1,0);
		\node at (-0.75,-0.75) {1};
		\node at (-0.75,-0.25) {3};
		\node at (-0.75,0.25) {2};
		\node at (0.25,0.75) {1};
		\node at (0.75,-0.25) {3};
         \node at (0.75,-0.75) {2};
        \node at (-0.25,-0.75) {3};
		\node at (-0.25,-0.25) {1};
		\node at (-0.25,0.25) {2};
	\end{tikzpicture}
\end{align*}
    \caption{A flagged filling.}
    \label{fig:enter-label-UU22}
\end{figure}
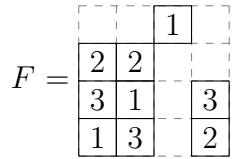
The column reading words are $231, 213, 1, 32$. So we have
\[
\mathrm{inv}(F)=2+1+0+1=4.
\]
For  $F\in \mathcal{F}_{D}$, we use    $y^{F}$ to denote the  following monomial in  $\mathbb{C} [Y]$ determined by $F$:
	\begin{align*}
		y^{F}=\prod_{(i, j)\in D}y_{c_{ij}i},
	\end{align*}
 where $c_{ij}$ is the entry in the box $(i, j)$ of $F$. The flagged  filling in   Figure \ref{fig:enter-label-UU22} gives
 \[
y^{F}=y_{11}\,y_{22}^2\,y_{13}\,y_{33}^2\,y_{14}\,y_{24}\,y_{34}.
 \]

 For   $C\leq D$,
 let $\mathcal{F}_{D}(C)$ denote the subset of  $\mathcal{F}_{D}$ consisting of the flagged fillings $F\in \mathcal{F}_{D}$  such that for $1\leq j\leq n$,  the set of entries in the $j$-th column of $F$ is exactly $C_j$.

 \begin{lem}\label{ABC-1}
For $C\leq D$,  we have
	\begin{align}\label{OOPP}	\mathrm{det}\left(Y_{D}^C\right)=\prod_{j=1}^{n}\mathrm{det}\left(Y_{D_{j}}^{C_{j}}\right)=\sum_{F\in \mathcal{F}_{D}(C)}\mathrm{sgn}(F)\cdot y^{F},
	\end{align}
	where $\mathrm{sgn}(F)=(-1)^{\mathrm{inv}(F)}$.
\end{lem}

\begin{proof}
Let $\mathcal{F}_{D_j}(C_j)$ denote the set of flagged fillings of $D_j$ with the numbers in $C_j$. Clearly,
\[
\mathcal{F}_{D}(C)=\mathcal{F}_{D_1}(C_1)\times \cdots \times \mathcal{F}_{D_n}(C_n).
\]
To verify 	\eqref{OOPP}, it suffices to check that
\begin{align}\label{OOPPQQ}
\mathrm{det}\left(Y_{D_{j}}^{C_{j}}\right)=\sum_{F_j\in \mathcal{F}_{D_j}(C_j)}\mathrm{sgn}(F_j)\cdot y^{F_j}.
	\end{align}
To see this, let $C_j=\{c_1<\cdots<c_m\}$ and $D_j=\{d_1<\cdots<d_m\}$. Then the nonzero momomials appearing  in the expansion of   $\mathrm{det}\left(Y_{D_{j}}^{C_{j}}\right)$ are
\[
(-1)^{\mathrm{inv}(w)}y_{w(c_1)d_1}\cdots y_{w(c_m)d_m},
\]
where $w$ can be taken as any permutation of $C=\{c_1,\ldots, c_m\}$ such  that $w(c_t)\leq d_t$ for $t=1,\ldots, m$.
Each such $w$ defines a flagged filling of $D_j$ by placing $w(c_t)$ to the box of $D_j$ in row $d_t$, and vice versa. This implies   \eqref{OOPPQQ}, as desired.
\end{proof}

\begin{re}
It should be noticed   that  cancellations possibly occur in the summation \eqref{OOPP}. This means that   there may exist fillings $F^{(1)},F^{(2)}\in \mathcal{F}_{D}(C)$ that contribute the same monomial, but have opposite signs.
See an example  illustrated in Figure  \ref{fig:enter-label-PUYTT}.
\begin{figure}[ht]
    \centering
    \begin{align*}
	\begin{tikzpicture}[baseline={([yshift=-.5ex]current bounding box.center)}]
		\centering
		\draw[dashed,step=0.5,help lines] (-2-5,-2) grid (1-5,1);
		\draw[black,step=0.5] (-2-5,-1) grid (-1-5,0);
		\draw[black,step=0.5] (-2-5,-2) grid (-1-5,-1.5);
		\draw[black,step=0.5] (0-5,-2) grid (0.5-5,-1.5);
		\draw[black,step=0.5] (0-5,-1) grid (0.5-5,0);
		\node at (-1.75-5,-1.75) {3};
		\node at (-1.75-5,-0.75) {2};
		\node at (-1.75-5,-0.25) {1};
		\node at (-1.25-5,-1.75) {1};
		\node at (-1.25-5,-0.75) {3};
		\node at (-1.25-5,-0.25) {2};
		\node at (0.25-5,-1.75) {2};
		\node at (0.25-5,-0.75) {1};
		\node at (0.25-5,-0.25) {3};
        \node at (-7.7,-0.5) {$F^{(1)}=$};
        \node at (-2.7,-0.5) {$F^{(2)}=$};
  	\draw[dashed,step=0.5,help lines] (-2,-2) grid (1,1);
		\draw[black,step=0.5] (-2,-1) grid (-1,0);
		\draw[black,step=0.5] (-2,-2) grid (-1,-1.5);
		\draw[black,step=0.5] (0,-2) grid (0.5,-1.5);
		\draw[black,step=0.5] (0,-1) grid (0.5,0);
        \node at (-1.75,-1.75) {1};
		\node at (-1.75,-0.75) {2};
		\node at (-1.75,-0.25) {3};
		\node at (-1.25,-1.75) {3};
		\node at (-1.25,-0.75) {1};
		\node at (-1.25,-0.25) {2};
		\node at (0.25,-1.75) {2};
		\node at (0.25,-0.75) {3};
		\node at (0.25,-0.25) {1};
	\end{tikzpicture}
\end{align*}
    \caption{Two fillings such that $y^{F^{(1)}}=y^{F^{(2)}}$ and $\mathrm{sgn}(F^{(1)})=-\mathrm{sgn}(F^{(2)})$.}
    \label{fig:enter-label-PUYTT}
\end{figure}
\end{re}

For   $F\in \mathcal{F}_{D}$, we define the weight of $F$ to be a vector
\begin{align*}
		\mathrm{wt}(F)=(f_{1},\ldots,f_{n}),
	\end{align*}	
where $f_{i}$ is the total sum of entries in the $i$-th row of $F$.
Recall that for two vectors $f=(f_1,\ldots, f_n)$ and $f'=(f_1',\ldots, f_n')$ of nonnegative integers, $f> f'$ in the lexicographical order if there exists an $1\leq i\leq n$ such that $f_i>f_i'$ and $f_t=f_t'$ for $1\leq t<i$.

\begin{lem}\label{ABC-2}
	For   $C\leq D$, there is a unique filling, denoted  $F_{\mathrm{max}}$, in
 $\mathcal{F}_{D}(C)$ with the maximum weight, that is,  $\mathrm{wt}(F_{\mathrm{max}})>\mathrm{wt}(F)$ for any $F\in \mathcal{F}_{D}(C)$ with $F\neq F_{\mathrm{max}}$.
\end{lem}

\begin{proof}
The proof is constructive. For $1\leq j\leq n$, the $j$-th column of $F_{\mathrm{max}}$ is defined as follows. Assume that $D_j=\{d_1<\cdots<d_m\}$. Fill the boxes of $D_j$ from top to bottom with the numbers in $C_j$ such that each box is assigned with a number as large as possible. Precisely,   the first box $(d_1, j)$ is filled with  the largest element, say $c_1$, of $C_j$
not exceeding $d_1$.
We obey the same rule to fill the second box $(d_2, j)$ with the numbers in $C_j\setminus \{c_1\}$.
Continuing the above procedure, we eventually obtain a flagged   filling in  $\mathcal{F}_{D}(C)$, which is defined as $F_{\mathrm{max}}$.
 See Figure \ref{fig:enter-label-PUYT} for an illustration of the construction.

\begin{figure}[ht]
    \centering
    \begin{align*}
	\begin{tikzpicture}[baseline={([yshift=-.5ex]current bounding box.center)}]
        \centering
  	\draw[dashed,step=0.5,help lines] (-2,-2) grid (1,1);
		\draw[black,step=0.5] (-2,-1.5) grid (-1.5,-1);
        \node at(-1.75,-1.25) {2};
        \draw[black,step=0.5] (-2,-0.5) grid (-1.5,0);
        \node at(-1.75,-0.25) {3};
        \draw[black,step=0.5] (-1.5,-1) grid (-1,-0.5);
        \node at(-1.25,-0.75) {3};
        \draw[black,step=0.5] (-1.5,0.5) grid (-1,1);
        \node at(-1.25,0.75) {1};
        \draw[black,step=0.5] (-1,0.5) grid (-0.5,-0.5);
        \node at(-0.75,-0.25) {1};
        \node at(-0.75,0.25) {2};
        \draw[black,step=0.5] (-1,-2) grid (-0.5,-1.5);
        \node at(-0.75,-1.75) {5};
        \draw[black,step=0.5] (-0.5,-2) grid (0,-0.5);
        \node at(-0.25,-1.75) {1};
        \node at(-0.25,-1.25) {3};
        \node at(-0.25,-0.75) {4};
        \draw[black,step=0.5] (-0.5,0) grid (0,0.5);
        \node at(-0.25,0.25) {2};
        \draw[black,step=0.5] (0,-1) grid (0.5,0);
        \node at(0.25,-0.75) {2};
        \node at(0.25,-0.25) {3};
        \draw[black,step=0.5] (0.5,-1.5) grid (1,-1);
        \node at(0.75,0.25) {1};
        \node at(0.75,-1.25) {4};
        \draw[black,step=0.5] (0.5,0) grid (1,0.5);
	\end{tikzpicture}
\end{align*}
    \caption{A flagged  filling with maximum weight.}    \label{fig:enter-label-PUYT}
\end{figure}

Suppose that $\mathrm{wt}(F_{\mathrm{max}})=(f_{1},\ldots,f_{n})$. Let $F$ be a filling in $\mathcal{F}_{D}(C)$ that has the same weight as $F_{\mathrm{max}}$. We need to check that $F=F_{\mathrm{max}}$.
Let us first consider the first row of $F$ (Without loss of generality, assume that the first row of $F$ is nonempty).
Let $(1,j)$ be any box of $D$ in the first row. By the construction of $F_{\mathrm{max}}$, the entry of $F$ in $(1,j)$   cannot exceed the entry of $F_{\mathrm{max}}$ in $(1,j)$. Since the sum of entries in the first row of $F$ is $f_1$,  $F$ and $F_{\mathrm{max}}$ must have the same entry in $(1,j)$.  The same analysis applies to the second row, and eventually we may conclude  that $F=F_{\mathrm{max}}$.
\end{proof}

We are now ready to prove the sufficiency of Theorem \ref{conj}.

\begin{theo}
Fix a diagram $D$. Let $x^a$ be any monomial appearing in
$\chi_D(x)$. If $D$ avoids the configuration as shown in
Figure \ref{2.1}, then the set
\begin{align*}
\left\{\mathrm{det}\left(Y_{D}^C\right)\colon  C\leq D, \, x^C=x^a\right\}
\end{align*}
of polynomials generating the eigenspace corresponding to $x^a$ is linearly independent.
\end{theo}

\begin{proof}

Suppose that $C^{(1)},\ldots,C^{(m)}$ are all the diagrams $C\leq D$ such that $x^{C}=x^a$.
For $1\leq i\leq m$, let $F_{\mathrm{max}}^{(i)}$ be the flagged filling in $\mathcal{F}_D(C^{(i)})$ with the maximum weight. We further assume that the weights of $\mathcal{F}_D(C^{(i)})$ are weakly decreasing, namely,
\begin{equation}\label{UYTR}
\mathrm{wt}(\mathcal{F}_D(C^{(1)}))\geq  \cdots \geq
\mathrm{wt}(\mathcal{F}_D(C^{(m)})).
\end{equation}
For simplicity, we denote
\[
\textbf{ y}^{(i)}=y^{F_{\mathrm{max}}^{(i)}}.
\]

\noindent
{\textbf{Claim 1.}}
For $1\leq i\leq m$, the monomial $\textbf{ y}^{(i)}$ appears in the polynomial
$\mathrm{det}(Y_D^{C^{(i)}})$.

To conclude this claim, by Lemma \ref{ABC-1}, it suffices  to verify $y^F\neq \textbf{ y}^{(i)}$  for any $F\in \mathcal{F}_D(C^{(i)})$
with $F\neq F_{\mathrm{max}}^{(i)}$. This follows from
 Lemma \ref{ABC-2} together with the fact that
if  $y^F=y^{F'}$  for $F, F'\in \mathcal{F}_D(C^{(i)})$, then $\mathrm{wt}(F)=\mathrm{wt}(F')$.

The second  claim plays a key role in the proof of the theorem, whose proof will use the subdiagram avoidance condition given in Figure \ref{2.1}.

\noindent
{\textbf{Claim 2.}} The monomials $\textbf{ y}^{(1)}, \ldots, \textbf{ y}^{(m)}$ are distinct.

We prove this claim by contradiction. Suppose to the contrary that there exist $1\leq i<j\leq m$ such that
$\textbf{ y}^{(i)}=\textbf{ y}^{(j)}$. Then, for any $1\leq r\leq n$, the multi-set of  entries in the $r$-th row of $F_{\mathrm{max}}^{(i)}$ is the same as the multi-set of  entries in the $r$-th row of $F_{\mathrm{max}}^{(j)}$.
(It should be noted  that $F_{\mathrm{max}}^{(i)}\neq F_{\mathrm{max}}^{(j)}$ since otherwise we would have $C^{(i)}=C^{(j)}$.) Thus $F_{\mathrm{max}}^{(i)}$ can be obtained from $F_{\mathrm{max}}^{(j)}$ by rearranging the entries in each row.

Locate any index  $1\leq r\leq n$ such that the $r$-th row  $R_r$  of $F_{\mathrm{max}}^{(i)}$ is different from the $r$-th row $R_r'$ of $F_{\mathrm{max}}^{(j)}$. Keep in mind that each entry in the $r$-th row cannot exceed $r$, and that $R'$ can be obtained  by rearranging the entries of $R'$.
There must exist  an entry, say $a$, such that $a<r$ and $a$ lies in distinct columns of $R_r$ and $R_r'$.  Suppose that $a$ lies the $k$-th (resp., $k'$-th) column of $R_r$ (resp., $R_r'$). Below is a crucial observation.
\begin{itemize}
\item Both the box $(a, k)$ and  the box $(a, k')$ do not belong to $D$.
\end{itemize}
This can be seen as follows. Suppose otherwise that $(a, k)$ (resp., $(a, k')$) belongs to $D$. Then, by the construction of the
$k$-th (resp., $k'$-th) column of   $F_{\mathrm{max}}^{(i)}$ (resp., $F_{\mathrm{max}}^{(j)}$), the box $(a, k)$  (resp., $(a, k')$)  in $F_{\mathrm{max}}^{(i)}$  (resp., $F_{\mathrm{max}}^{(j)}$)
would  be filled with $a$, leading to a contradiction.

The above observation yields that the boxes $(r, k)$, $(r, k')$, $(a, k)$, $(a, k')$ form  a configuration as shown in Figure \ref{2.1}. Hence the assumption that $\textbf{ y}^{(i)}=\textbf{ y}^{(j)}$ is false. This finishes the proof of Claim 2.

We use Claims 1 and 2 to show that
$\mathrm{det}\left(Y_{D}^{C^{(1)}}\right), \ldots, \mathrm{det}\left(Y_{D}^{C^{(m)}}\right)$
are linearly independent. The idea  is similar to that of \cite[Thm 2.5.8]{2001The}. Assume that we have the following linear relation
\begin{align*}
		c_1\mathrm{det}(Y_{D}^{C^{(1)}})+c_2\mathrm{det}(Y_{D}^{C^{(2)}})+\cdots+c_n\mathrm{det}(Y_{D}^{C^{(m)}})=0.
\end{align*}
Combining  the assumption in \eqref{UYTR}, Claim 1 (together with its proof) and Claim 2, we see that the nomomial $\textbf{ y}^{(1)}$ only appears in $\mathrm{det}(Y_{D}^{C^{(1)}})$.
This leads to $c_1=0$. Using the same analysis, we conclude $c_i=0$ for $i=2,\ldots, m$. This completes the proof.
\end{proof}

\section{Necessity of Theorem  \ref{conj}}

The following theorem confirms the necessity of Theorem  \ref{conj}.

\begin{theo}
If the diagram $D$ has the configuration shown in Figure \ref{2.1}, then   the set
\begin{align*}
\left\{\mathrm{det}\left(Y_{D}^C\right)\colon  C\leq D\right\}
\end{align*}
of polynomials generating the module $\mathcal{M}_{D}$ is linearly dependent.
\end{theo}

\begin{proof}
Let us locate  (any) two columns, say columns $j_1$ and $j_2$, with $1\leq j_1<j_2\leq n$, that owe  a configuration  as shown in Figure \ref{2.1}. Of course, boxes in these two columns may consist of more than one such configurations. In this case, we choose a ``minimal'' one, that is, locate   two row indices  $1\leq i_{1}<i_2\leq n$  such that
\begin{itemize}
    \item the boxes $(i_1,j_{1}),(i_1,j_{2})$ do not belong to $D$;

    \item the boxes $(i_2,j_{1}),(i_2,j_{2})$  belong to $D$;

    \item for any   $i_{1}<i<i_{2}$, exactly one of the boxes  $(i,j_{1})$ and $(i,j_{2})$ belongs to $D$.
\end{itemize}
See  Figure    \ref{fig:enter-label-OOOII} for an illustration.
     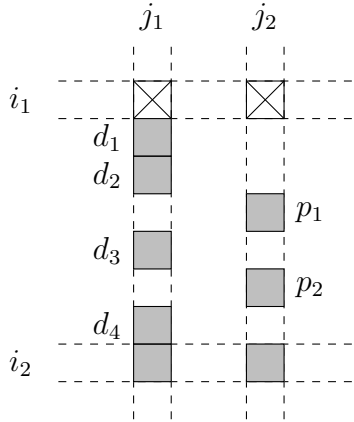
\begin{figure}[ht]
         \centering
         \begin{align*}
		\begin{tikzpicture}
			\centering
			\draw[dashed, black] (-2,3) -- (2,3);
			\draw[dashed, black] (-2,2.5) -- (2,2.5);
			\draw[dashed, black] (-2,-1) -- (2,-1);
			\draw[dashed, black] (-2,-0.5) -- (2,-0.5);
			\draw[dashed, black] (-0.5,-1.5) -- (-0.5,3.5);
			\draw[dashed, black] (-1,-1.5) -- (-1,3.5);
			\draw[dashed, black] (1,-1.5) -- (1,3.5);
			\draw[dashed, black] (0.5,-1.5) -- (0.5,3.5);
			\draw[black] (-1,2.5) rectangle (-0.5,3);
            \filldraw[lightgray] (-1,2) rectangle (-0.5,2.5);
            \filldraw[lightgray] (-1,1.5) rectangle (-0.5,2);
            \filldraw[lightgray] (-1,0.5) rectangle (-0.5,1);
            \draw[black] (-1,0.5) rectangle (-0.5,1);
            \draw[black] (-1,2) rectangle (-0.5,2.5);
            \draw[black] (0.5,2.5) rectangle (1,3);
			\draw[black] (-1,2.5) -- (-0.5,3);
			\draw[black] (0.5,2.5) -- (1,3);
			\draw[black] (-1,3) -- (-0.5,2.5);
			\draw[black] (0.5,3) -- (1,2.5);
			\draw[black] (-1,1.5) rectangle (-0.5,2);
			\filldraw [lightgray] (-1,-1) rectangle (-0.5,-0.5);
			\filldraw [lightgray] (-1,-0.5) rectangle (-0.5,0);
			\draw[black] (-1,-1) rectangle (-0.5,-0.5);
			\draw[black] (-1,-0.5) rectangle (-0.5,0);
			\filldraw [lightgray] (0.5,-1) rectangle (1,-0.5);
			\draw[black] (0.5,-1) rectangle (1,-0.5);
            \filldraw [lightgray] (0.5,1) rectangle (1,1.5);
            \filldraw [lightgray] (0.5,0) rectangle (1,0.5);
            \draw[black] (0.5,1) rectangle (1,1.5);
            \draw[black] (0.5,0) rectangle (1,0.5);
			\node at (-2.5,2.75) {$i_{1}$};
			\node at (-2.5,-0.75) {$i_{2}$};
            \node at (-1.35,2.25) {$d_{1}$};
            \node at (-1.35,1.75) {$d_{2}$};
            \node at (-1.35,0.75) {$d_{3}$};
            \node at (-1.35,-0.25) {$d_{4}$};
            \node at (1.35,1.25) {$p_{1}$};
            \node at (1.35,0.25) {$p_{2}$};
			\node at (-0.75,3.85) {$j_{1}$};
			\node at (0.75,3.85) {$j_{2}$};
		\end{tikzpicture}
	\end{align*}
         \caption{An instance of  minimal configuration in Figure \ref{2.1}.}
         \label{fig:enter-label-OOOII}
     \end{figure}
	
We shall next construct a family of  polynomials $\mathrm{det}\left(Y_{D}^C\right)$ that are linearly dependent. Let us partition  $D_{j_1}$ into
\[
D_{j_{1}}=D_{j_{1}}^{(1)}\cup D_{j_{1}}^{(2)}\cup D_{j_{1}}^{(3)},
\]
where $D_{j_{1}}^{(1)}$ (resp., $D_{j_{1}}^{(3)}$)  is the subset of elements of $D_{j_1}$ that are smaller
than $i_1$
(resp., greater than $i_2$), and $D_{j_{1}}^{(2)}$ consists of the elements of $D_{j_1}$ lying between $i_1$ and $i_2$. We similarly partition $D_{j_2}$ into
\[
D_{j_{2}}=D_{j_{2}}^{(1)}\cup D_{j_{2}}^{(2)}\cup D_{j_{2}}^{(3)}.
\]
Assume that
\begin{align*}
		D_{j_{1}}^{(2)}=\{d_{1}<\cdots <d_{s}<i_{2}\} \ \ \ \text{and}\ \ \
		D_{j_{2}}^{(2)}=\{p_{1}<\cdots<p_{t}<i_{2}\}.
\end{align*}

Consider a family of diagrams $C\leq D$ that satisfy the following conditions:
\begin{itemize}
    \item[(1)] $C_{j}=D_{j}$ for  $j\neq j_{1}$ and $j\neq j_{2}$;

    \item[(2)] $C_{j_{1}}=C_{j_{1}}^{(1)}\cup C_{j_{1}}^{(2)}\cup C_{j_{1}}^{(3)}$, where  $C_{j_{1}}^{(1)}=D_{j_{1}}^{(1)}$, $C_{j_{1}}^{(3)}=D_{j_{1}}^{(3)}$,  $C_{j_{1}}^{(2)}\leq D_{j_{1}}^{(2)}$ with the constraint that  the smallest  element of $C_{j_{1}}^{(2)}$ is greater or equal to  $i_{1}$;

    \item[(3)] $C_{j_{2}}=C_{j_{2}}^{(1)}\cup C_{j_{2}}^{(2)}\cup C_{j_{2}}^{(3)}$, where  $C_{j_{2}}^{(1)}=D_{j_{2}}^{(1)}$, $C_{j_{2}}^{(3)}=D_{j_{2}}^{(3)}$,  $C_{j_{2}}^{(2)}\leq D_{j_{2}}^{(2)}$ with the constraint that  the least element of $C_{j_{2}}^{(2)}$ is greater or equal to  $i_{1}$.
\end{itemize}
 For a diagram $C$ satisfying the above conditions, we see that
	\begin{align*}
\mathrm{det}\left(Y_{D}^{C}\right)&=\mathrm{det}	\left(Y_{D_{j_{1}}}^{C_{j_{1}}}\right)\cdot \mathrm{det}\left(Y_{D_{j_{2}}}^{C_{j_{2}}}\right)\cdot \prod_{j\neq j_1, j_2}\mathrm{det}\left(Y_{D_{j}}^{C_{j}}\right)\\[5pt]
		&=\mathrm{det}\left(Y_{D_{j_{1}}^{(2)}}^{C_{j_{1}}^{(2)}}\right)
		\cdot\mathrm{det}\left(Y_{D_{j_{2}}^{(2)}}^{C_{j_{2}}^{(2)}}\right)
		\cdot\mathrm{det}\left(Y_{D_{j_{1}}^{(1)}}^{C_{j_{1}}^{(1)}}\right)
		\cdot\mathrm{det}\left(Y_{D_{j_{2}}^{(1)}}^{C_{j_{2}}^{(1)}}\right)
		\\[5pt] &\ \ \ \ \cdot\mathrm{det}\left(Y_{D_{j_{1}}^{(3)}}^{C_{j_{1}}^{(3)}}\right)
		\cdot\mathrm{det}\left(Y_{D_{j_{2}}^{(3)}}^{C_{j_{2}}^{(3)}}\right)
		\cdot\prod_{j\neq j_1, j_2}\mathrm{det}\left(Y_{D_{j}}^{C_{j}}\right).
	\end{align*}
The last five factors in the above equality   are fixed determinants.
We shall choose distinct  $C_{j_{1}}^{(2)}\leq D_{j_{1}}^{(2)}$ and $C_{j_{2}}^{(2)}\leq D_{j_{2}}^{(2)}$ such that the products $\mathrm{det}\left(Y_{D_{j_{1}}^{(2)}}^{C_{j_{1}}^{(2)}}\right)
	\cdot\mathrm{det}\left(Y_{D_{j_{2}}^{(2)}}^{C_{j_{2}}^{(2)}}\right)$ are linearly dependent,  so that the corresponding polynomials $\mathrm{det}\left(Y_{D}^C\right)$ are linearly dependent.
	
By the choices of  $i_{1}$ and $i_{2}$, we have
	\begin{align*}
		D_{j_{1}}^{(2)}\cap D_{j_{2}}^{(2)}=\{i_{2}\}
	\end{align*}
	and
	\begin{align*}
		D_{j_{1}}^{(2)}\cup D_{j_{2}}^{(2)}=\{i_{1}+1,i_{1}+2,\ldots,i_{2}\}.
	\end{align*}
Consider the square matrix:
	\begin{align*}
		M=
		\begin{bmatrix}
			y_{i_{1},i_{2}} & y_{i_{1},i_{1}+1} & y_{i_{1},i_{1}+2} & \cdots &y_{i_{1},i_{2}}\\
			y_{i_{1}+1,i_{2}} & y_{i_{1}+1,i_{1}+1} & y_{i_{1}+1,i_{1}+2} & \cdots &y_{i_{1}+1,i_{2}}\\
			y_{i_{1}+2,i_{2}} & 0 & y_{i_{1}+2,i_{1}+2} & \cdots &y_{i_{1}+2,i_{2}}\\
			\vdots &\vdots&\vdots&\ddots&\vdots\\
			y_{i_{2},i_{2}} & 0 & 0 & \cdots &y_{i_{2},i_{2}}\\
		\end{bmatrix}.
	\end{align*}
Since the first and  the last columns   are the same, we have $\det(M)=0$. We next apply the  Laplace expansion to $M$ along the columns indexed by $D_{j_{1}}^{(2)}=\{d_{1},d_{2},\ldots, d_{s},i_{2}\}$ (Here, $i_2$ indicates the last column). Notice that the remaining columns of $M$ are indexed by $D_{j_{2}}^{(2)}=\{p_{1},\ldots,p_{t},i_{2}\}$  (Here, $i_2$ indicates the first  column).
Applying the  Laplace expansion, we obtain
\begin{align}\label{78uu}
\det(M)=\sum_{R}
	 \mathrm{sgn}(R)\cdot	 \det\left(M_{D_{j_{1}}^{(2)}}^R\right)\cdot
  \det\left(M_{D_{j_{2}}^{(2)}}^{\overline{R}}\right)=0,
\end{align}
where $R$ is a subset of $\{i_1,  i_1+1, \ldots, i_2\}$ with $s+1$ elements, $\overline{R}=\{i_1,  i_1+1, \ldots, i_2\}\setminus R$, and $\mathrm{sgn}(R)$ is a sign $\pm 1$  determined by $S$.
It is easily seen that
\[
\det\left(M_{D_{j_{1}}^{(2)}}^R\right)=\det\left(Y_{D_{j_{1}}^{(2)}}^R\right)
\]
and
\[
\det\left(M_{D_{j_{2}}^{(2)}}^{\overline{R}}\right)=(-1)^t\det\left(Y_{D_{j_{2}}^{(2)}}^{\overline{R}}\right).
\]
Hence,  \eqref{78uu} can be rewritten as
\begin{align} \label{TREQ}
\det(M)=\sum_{R} (-1)^t
	 \mathrm{sgn}(R)\cdot	 \det\left(Y_{D_{j_{1}}^{(2)}}^R\right)\cdot
  \det\left(Y_{D_{j_{2}}^{(2)}}^{\overline{R}}\right)=0.
\end{align}

Set $C_{j_{1}}^{(2)}=R$ and $C_{j_{2}}^{(2)}=\overline{R}$. As $R$ varies over the sum in  \eqref{TREQ}, we
get a family of polynomials $\det\left(Y_{D_{j_{1}}^{(2)}}^R\right)\cdot
  \det\left(Y_{D_{j_{2}}^{(2)}}^{\overline{R}}\right)$ that are linearly dependent. Thus the corresponding diagrams $C\leq D$, that satisfy the conditions (1), (2) and (3), are a family of diagrams such that $\det\left(Y_{D}^C\right)$   are linearly dependent.
This completes the proof.
\end{proof}

\vskip 3mm \noindent {\bf Acknowledgments.}
This work was supported by the National Natural Science Foundation of China (11971250, 12001398).

\bigbreak

\footnotesize{

\textsc{(Zhuowei Lin) Center for Combinatorics, Nankai University, Tianjin 300071, P.R. China.}

{\it
Email address: \tt zwlin0825@163.com}

\medbreak

\textsc{(Simon C.Y. Peng) Center for Applied Mathematics, Tianjin University, Tianjin 300072, P.R. China.}

{\it
Email address: \tt pcy@tju.edu.cn}

\medbreak

\textsc{(Sophie C.C. Sun) Department  of  Mathematics, University  of  Finance  and  Economics, Tianjin  300222, P.R. China.}

{\it
Email address: \tt sophiesun@tjufe.edu.cn.}

}


\begin{thebibliography}{1}

\bibitem{BH}
P. Br\"and\'en and J. Huh,  Lorentzian polynomials, Ann. of Math.   192 (2020),   821--891.

\bibitem{FG-1}
N.J.Y. Fan and P.L. Guo,
Vertices of Schubitopes,
J. Combin. Theory Ser. A  177 (2021), 105311.

\bibitem{fan2022upper}
N.J.Y. Fan and P.L. Guo,
Upper bounds of schubert polynomials,
Sci. China Math. 65 (2022),  1319--1330.

\bibitem{2018Schubert}
A. Fink, K. M{\'e}sz{\'a}ros and A. St. Dizier,
Schubert polynomials as integer point transforms of generalized permutahedra,
Adv. Math. 332 (2018), 465--475.

\bibitem{fink2021zero}
A. Fink, K. M{\'e}sz{\'a}ros and A. St. Dizier,
Zero-one Schubert polynomials,
Math. Z. 297 (2021),   1023--1042.

\bibitem{hodges2023multiplicity}
R. Hodges and A. Yong,
Multiplicity-free key polynomials,
Ann. Comb. 27 (2023),  387--411.

\bibitem{Huh}
J. Huh, J.P. Matherne,  K. M{\'e}sz{\'a}ros  and A. St. Dizier, Logarithmic concavity of Schur and related polynomials, Trans. Amer. Math. Soc. 375 (2022), 4411--4427.


\bibitem{kraskiewicz1987foncteurs}
W. Kraskiewicz and P. Pragacz,
Foncteurs de schubert,
C. R. Acad. Sci. Paris S{\'e}r. I Math. 304 (1987),  209--211.

\bibitem{kraskiewicz2004schubert}
W. Kra{\'s}kiewicz and P. Pragacz,
Schubert functors and Schubert polynomials,
European J. Combin. 25 (2004),  1327--1344.

\bibitem{magyar1998schubert}
P. Magyar,
Schubert polynomials and Bott--Samelson varieties,
Comment. Math. Helv. 73 (1998), 603--636.

\bibitem{meszaros2021principal}
K. M{\'e}sz{\'a}ros, A. St. Dizier and A. Tanjaya,
Principal specialization of dual characters of flagged weyl modules,
Electron, J. Combin. 28 (2021),   Paper No. 4.17, 12 pp.

\bibitem{2001The}
B.E. Sagan, The symmetric group. Representations, combinatorial algorithms, and symmetric functions. Second edition, Graduate Texts in Mathematics, 203,  Springer-Verlag, New York, 2001.

\end{thebibliography}
\end{document}